\documentclass{amsart}

\usepackage{amsmath, amssymb}
\usepackage{amsthm}
\usepackage{enumerate}
\usepackage{mathtools}

\usepackage{tikz-cd}

\usepackage[shortlabels]{enumitem}
\setlist[enumerate]{label=\rm{(\arabic*)}}
\setlist[enumerate,2]{label=\rm({\it\roman*})}
\setlist[itemize]{label=\raisebox{0.25ex}{\tiny$\bullet$}}
\usepackage[backref, colorlinks, linktocpage, citecolor = blue, linkcolor = blue]{hyperref}

\newcommand{\p}{\operatorname{\mathbb{P}}}
\newcommand{\C}{\operatorname{\mathbb{C}}}
\newcommand{\Z}{\operatorname{\mathbb{Z}}}

\newcommand{\Ind}{\operatorname{Ind}}
\newcommand{\A}{\operatorname{\mathbb{A}}}
\newcommand{\Q}{\operatorname{\mathbb{Q}}}

\newcommand{\F}{\operatorname{\mathbb{F}}}
\newcommand{\dpol}{\operatorname{dpol}}

\newcommand{\car}{\operatorname{char}}

\newcommand{\aut}{\operatorname{Aut}}
\newcommand{\End}{\operatorname{End}}

\newcommand{\Bir}{\operatorname{Bir}}

\newcommand{\exc}{\operatorname{Exc}}

\newcommand{\dom}{\operatorname{dom}}
\newcommand{\Div}{\operatorname{Div}}
\newcommand{\Rat}{\operatorname{Rat}}

\theoremstyle{plain}
\newtheorem{theorem}{Theorem}
\newtheorem{lemma}[theorem]{Lemma}
\newtheorem{proposition}[theorem]{Proposition}
\newtheorem{corollary}[theorem]{Corollary}
\theoremstyle{definition}
\newtheorem*{definition}{Definition}
\newtheorem{example}[theorem]{Example}
\theoremstyle{remark}

\newtheorem{question}{Question}

\numberwithin{theorem}{section}
\title{Remarks on the degree growth of birational transformations}
\author{Christian Urech}

\date{June 2016}
\address{Mathematisches Institut\\ Universit\"at Basel\\ 4051 Basel\\ Switzerland}
\address{IRMAR\\ Universit\'e de Rennes 1\\ 35042 Rennes\\ France}
\email{christian.urech@unibas.ch}

\thanks{The author gratefully acknowledges support by the Swiss National Science Foundation Grant "Birational Geometry"  PP00P2 128422 /1 as well as by the Geldner-Stiftung, the FAG Basel, the Janggen P\"ohn-Stiftung and the State Secretariat for Education,
Research and Innovation of Switzerland}

\begin{document}
\maketitle

\begin{abstract}
We look at sequences of positive integers that can be realized as degree sequences of iterates of rational dominant maps of smooth projective varieties over arbitrary fields. New constraints on the degree growth of endomorphisms of the affine space and examples of degree sequences are displayed. We also show that the set of all degree sequences of rational maps is countable; this generalizes a result of Bonifant and Fornaess.
\end{abstract}

\tableofcontents

\section{Introduction and results}
\subsection{Groups of birational transformations and degree sequences}
Let $X_k$ be a projective variety defined over a field $k$; denote by $\Bir(X_k)$ the group of birational transformations of $X_k$. A group $\Gamma$ is called a {\it group of birational transformations} if there exists a field $k$ and a projective variety $X_k$ over $k$ such that $\Gamma\subset\Bir(X_k)$. More generally, one can consider  $\Rat(X_k)$, the monoid of dominant rational maps of $X_k$. Accordingly, we call a monoid $\Delta$ a {\it monoid of rational dominant transformations}, if there exists a field $k$ and a projective variety $X_k$ over $k$ such that $\Delta\subset\Bir(X_k)$. 

If $X_k$ is a smooth projective variety, an interesting tool to study the structure of monoids of rational dominant transformations are {\it degree functions}. Fix a polarization of $X_k$, i.e. an ample divisor class $H$ of $X_k$. Then one can associate to every element $f\in\Rat(X_k)$ its degree $\deg_H(f)\in\Z^+$ with respect to $H$, which is defined by
\[
\deg_H(f)=f^*H\cdot H^{d-1},
\]
where $d$ is the dimension of $X_k$ and $f^*H$ is the total transform of $H$ under $f$. For a smooth projective variety $X_k$ over a field of characteristic zero $k$, one has for $f,g\in\Bir(X_k)$
\[
\deg_H(f\circ g)\leq C(X_k, H)\deg_H(f)\deg_H(g),
\]
where $C(X_k,H)$ is a constant only depending on $X_k$ and the choice of polarization $H$ (see \cite{MR2180409}). For a generalization of this result to fields of positive characteristic, see \cite{Truong:2015aa}.

Let $f$ be a rational self map of $\p_k^d$. With respect to homogeneous coordinates $[x_0:\dots:x_d]$ of $\p_k^d$, $f$ is given by $[x_0:x_1:\dots:x_d]\mapsto [f_0:f_1:\dots:f_d]$, where $f_0,\dots, f_d\in k[x_0,\dots,x_d]$ are homogeneous polynomials of the same degree and without a common factor. We define $\deg(f)=\deg(f_i)$. Note that if $f$ is dominant, then $\deg(f)=\deg_H(f)$ for $H=\mathcal{O}(1)$. So in case $X_k=\p_k^d$ we can  extend the notion of degree to all rational self maps. Note that if $f$ is an endomorphism of $\A_k^d$ defined by $(x_1,\dots, x_d)\mapsto (f_1,\dots, f_d)$ with respect to affine coordinates $(x_1,\dots, x_d)$ of $\A_k^d$, where $f_1,\dots, f_d\in k[x_1,\dots, x_d]$,  then $\deg(f)$ is the maximal degree of the $f_i$. 

Let $\Delta\subset\Bir(X_k)$ be a finitely generated monoid of rational dominant transformations of a smooth projective variety $X_k$ with a finite set of generators $S$. We define 
\[
D_{S, H}\colon\Z^+\to\Z^+\]
by
 \[
D_{S,H}(n):=\max_{\delta\in B_S(n)}\{\deg_H(\gamma)\},
\]
where, $B_S(n)$ denotes all elements in $\Delta$ of word length $\leq n$ with respect to the generating set $S$. We call a map $\Z^+\to\Z^+$ that can be realized for some field $k$ and some $(X_k, H, \Delta, S)$ as such a function a {\it degree sequence}. 

Note that our definition of degree sequences includes in particular degree sequences that are given by finitely generated groups of birational transformations $\Gamma\subset\Bir(X_k)$.

In this paper we show that there exist only countably many degree sequences, display certain constraints on their growth and give some new examples. 

\subsection{Countability of degree sequences}
In \cite{MR1793690}, Bonifant and Fornaess proved that the set of sequences $\{d_n\}$ such that there exists a rational self map $f$ of $\p^d_{\C}$ satisfying $\deg(f^n)=d_n$ for all $n$, is countable, which answered a question of Ghys. We generalize the result of Bonifant and Fornaess to all degree sequences over all smooth projective varieties, all fields, all polarizations and all finite generating sets $S$ of finitely generated monoids of rational dominant maps:

\begin{theorem}\label{rationalcount}
The set of all degree sequences is countable.
\end{theorem}

\subsection{Previous results}
In dimension 2 the degree growth of birational transformations is well understood and is a helpful tool to understand the group structure of $\Bir(S_k)$ for projective surfaces $S_k$ over a field $k$. 

\begin{theorem}[Gizatullin; Cantat; Diller and Favre]\label{dim2}
Let $k$ be an algebraically closed field, $S_k$ a projective surface over $k$ with a fixed polarization $H$ and $f\in\Bir(S_k)$. Then one of the following is true:
\begin{itemize}
\item the set $\{\deg_H(f^n)\}$ is bounded;
\item $\deg_H(f^n)\sim cn$ for some positive constant $c$ and $f$ preserves a rational fibration;
\item $\deg_H(f^n)\sim cn^2$ for some positive constant $c$ and $f$ preserves an elliptic fibration;
\item $\deg_H(f^n)\sim \lambda^n $, where $\lambda$ is a Pisot or Salem number.
\end{itemize}
\end{theorem}

For more details on this rich subject and references to the proof of Theorem \ref{dim2} see \cite{cantat2012cremona}. 

In the case of polynomial automorphisms of the affine plane, the situation is even less complicated. Let $f\in\aut(\A_k^2)$. Then the sequence $\{\deg(f^n)\}$ is either bounded or it grows exponentially in $n$. See  \cite{MR1667603} for this and more results on the degree growth in $\aut(\A_k^2)$

In higher dimensions there are only few results on the degree growth of birational transformations. In particular, the following questions are open:
\begin{question}
Does there exist a birational transformation $f$ of a projective variety $X_k$ such that $\deg_H(f^n)$ is of intermediate growth, for instance $\deg_H(f^n)\sim e^{\sqrt{n}}$?
\end{question}

\begin{question}
Does there exist a birational transformation $f$ such that $\deg_H(f^n)$ is unbounded, but grows ''slowly``? For instance, can we have $\deg_H(f^n)\sim \sqrt{n}$? Or do unbounded degree sequences grow at least linearly?
\end{question}

\begin{question}
If there is a birational transformation $f$ such that $\deg_H(f^n)\sim \lambda^n$, is $\lambda$ always an algebraic number?
\end{question}

\begin{question}
Do birational transformations of polynomial growth always preserve some non-trivial rational fibration?
\end{question}

 In \cite{MR3210135} Lo Bianco treats the case of automorphisms of compact K\"ahler threefolds.

\subsection{Degree sequences of polynomial automorphisms}
A good place to start the examination of degree sequences seems to be the group of polynomial automorphisms of the affine $d$-space $\aut(\A_k^d)$. In Section \ref{proofaut1} we will show the following observation (the proof of which can be found as well in \cite{Deserti:2016zl}):

\begin{proposition}\label{aut1}
Let $k$ be a field and $f\in\aut(\A^d_k)$ a polynomial automorphism such that $\deg(f^d)=\deg(f)^d$, then $\deg(f^n)=\deg(f)^n$ for all $n\in\Z^+$.
\end{proposition}

The monoid $\End(\A^d_k)$ has the additional structure of a $k$-vector space, on which the degree function induces a filtration of finite dimensional vector spaces. This gives rise to a new technique, which we will use to prove that unbounded degree sequences of groups of polynomial automorphisms diverge and can not grow arbitrarily slowly:

\begin{theorem}\label{degaut}
Let $f\in\End(\A_k^d)$ be an endomorphism such that the sequence $\{\deg(f^n)\}$ is unbounded. Then for all integers $K$
\[
\#\left\{m\mid \deg(f^m)\leq K\right\}< C_d\cdot K^{d},
\]
where $C_d=\frac{(1+d)^{d}}{(d-1)!}$. In particular, $\deg(f^n)$ diverges to~$\infty$.
\end{theorem}

By a result of Ol'shanskii (\cite{MR1714850}), Theorem \ref{degaut} shows that an unbounded degree sequence of a polynomial automorphism behaves in some ways like a word length function. 

The following corollary is immediate:
\begin{corollary}\label{degautcor}
Let $\Gamma\subset\End(\A^d_k)$ be a finitely generated monoid with generating system $S$. If $D_S(n)< C_d\cdot n^{1/d}$ for infinitely many $n$ then $\Gamma$ is of bounded degree.
\end{corollary}

Unfortunately our methods to prove Theorem \ref{degaut} do not work for arbitrary birational transformations of $\p^d_k$. However, if we assume the ground field to be finite, we obtain a similar result:

\begin{theorem}\label{degfinite}
Let $\F_q$ be a finite field with $q$ elements and let $f\in\Rat(\p^d_{\F_q})$ such that the sequence $\{\deg(f^n)\}$ is unbounded. Then, for all integers $K$,
\[
\#\left\{m\mid \deg(f^m)\leq K\right\}\leq q^{C(K,d)},
\]
where $C(K,d)=(d+1)\cdot\binom{d+K}{K}$. In particular, $\deg(f^n)$ diverges to $\infty$.
\end{theorem}

This implies the following for degree sequences in $\Rat(\p^d_{\F_q})$:

\begin{corollary}\label{maincor}
Let $\Gamma\subset\Rat(\p^d_{\F_q})$ be a finitely generated monoid with generating system $S$. There exists a positive constant $C_{d,q}$ such that if $D_S(n)<C_{d,q}\cdot\log(n)^{1/d}$ for all $n$, then $\Gamma$ is of bounded degree.
\end{corollary}

\subsection{Types of degree growth}

\begin{definition}
Let $X_k$ be a smooth projective variety with polarization $H$ over a field $k$ and let $f\in\Bir(X_k)$. We denote the {\it order of growth} of $\deg_H(f^n)$ by
\[
\dpol(f):=\limsup_{n\to\infty}\frac{\log(\deg_H(f^n))}{\log(n)}.
\]
The order of growth can be infinite.
\end{definition}

By a result of Dinh and Sibony, the order of growth does not depend on the choice of polarization if we work over the field of complex numbers (see Section~\ref{complex}):
 
\begin{proposition}\label{indeppolar}
Let $X_{\C}$ be a smooth complex projective variety  and let $f\in\Bir(X_{\C})$. Then $\dpol(f)$ does not depend on the choice of polarization. 
\end{proposition}

Let $f$ be a birational transformation of a surface. As recalled above, in that case $\dpol(f)=0,1, 2$ or $\infty$. This gives rise to the following question:

\begin{question}
Does there exist a constant $C(d)$ depending only on $d$ such that for all varieties $X_k$ of dimension $d$ we have $\dpol(f)<C(d)$ for all $f\in \Bir(X_k)$ with $\dpol(f)$ finite?
\end{question}

We give some examples of degree sequences that indicate that the degree growth in higher dimensions is richer than in dimension 2. 

First of all, note that elements in $\aut(\A^d_k)$ can have polynomial growth:

\begin{example}\label{linearex}
Let $k$ be any field and define $f,g,h\in\aut(\A^d_k)$ by $g=(x+yz, y, z)$, $h=(x, y+xz, z)$ and  
\[
f=g\circ h=(x+z(y+xz), y+xz,z).
\]
One sees by induction that $\deg(f^n)=2n+1$; in particular, $\dpol(f)=1$.
\end{example}

\begin{example}\label{exaut}
More generally, for all $l\leq d/2$ there exist elements  $f_l\in\aut(\A_k^d)$ such that $\dpol(f)=l$ (Section \ref{proofexaut}). 
\end{example}

Other interesting examples of degree sequences of polynomial automorphisms and the dynamical behavior of the corresponding maps are described in~\cite{Deserti:2016zl}.

For birational transformations of $\p^d_k$ we can obtain even faster growth (see \cite{MR2917145} for more details):
\begin{example}
The birational transformation $f=(x_1, x_1x_2,\dots, x_1x_2\cdots x_n)$ of $\p_k^d$ defined with respect to local affine coordinates $(x_1,\dots, x_d)$ satisfies {$\deg(f^n)=n^{d-1}$}, i.e. $\dpol({f})=d-1$.
\end{example}

The following interesting observation is due to Serge Cantat:
\begin{example}\label{oguiso}
Define $\omega:=\frac{-1+\sqrt{-3}}{2}$ and the elliptic curve $E_\omega:=\C/(\Z+\Z\omega)$. Let 
\[
X:=E_\omega\times E_\omega\times E_\omega
\]
 and $s\colon X\to X$ the automorphism of finite order given by diagonal multiplication with $-\omega$. In \cite{MR3329200} Oguiso and Truong prove that the quotient $Y:=X/s$ is a rational threefold. Let $f\colon X\to X$ be the automorphism defined by $(x_1,x_2,x_3)\mapsto (x_1, x_1+x_2, x_2+x_3)$. Since $f$ commutes with $s$, it induces an automorphism on $Y$, which we denote by $\hat{f}$. Let $\phi_1\colon\tilde{Y}\to Y$ be a resolution of the singularities of $Y$ and define $\tilde{f}\in\Bir(\tilde{Y})$ by
 \[
  \tilde{f}:=\phi_1^{-1}\circ\hat{f}\circ\phi_1.
  \]
  We will show in Section \ref{oguisoproof} that $\dpol(\tilde{f})=4$.
\end{example}

\subsection{Acknowledgements}
I thank my advisors J\'er\'emy Blanc and Serge Cantat for many interesting and helpful discussions and their constant support. I would also like to thank Junyi Xie for inspiring conversations and for giving me, together with Serge Cantat, some of the main ideas for the proof of Theorem \ref{rationalcount}, as well as Federico Lo Bianco for showing me Proposition~\ref{propfede}. Many thanks also to Bac Dang and Julie D\'eserti for their comments and helpful references.

\section{Preliminaries} 
\subsection{Monoids of rational dominant transformations}
Let $X_k$ be a projective variety over a field~$k$. There is a one to one correspondence between rational dominant self maps of $X_k$  and $k$-endomorphisms of the function field $k(X_k)$. The field $k(X_k)$ is the field of fractions of a $k$-algebra of finite type $k[x_1,\dots, x_n]/I$, where $I\subset k[x_1,\dots, x_n]$ is a prime ideal generated by  elements $f_1,\dots, f_l\in k[x_1,\dots, x_n]$. A field extension $k\to k'$ induces a base change $X_{k'}\to X_k$. The function field of $X_{k'}$ is the field of fractions of the $k'$-algebra $k'[x_1,\dots,x_n]/I'$, where $I'$ is the ideal generated by $f_1,\dots, f_l$. Note that $k(X_k)\subset k'(X_{k'})$. We say that a $k'$-endomorphism of $k'(X_{k'})$ is defined over $k$, if it restricts to a $k$-endomorphism of $k(X_k)$. Consider a $k'$-endomorphism $G$ of $k'(X_{k'})$ sending generators $(x_1,\dots, x_n)$ of $k'(X_{k'})$ to $(g_1,\dots, g_n)$, where $g_i\in k'(X_{k'})$. Then $G$ is defined over $k$ if and only if $g_i\in k(X_k)$ for all $i$. On the other hand, let $g_i,\dots, g_n\in k(X_k)$ and let $(x_1,\dots, x_d)\mapsto(g_1,\dots, g_d)$ be a $k$-endomorphism of $k(X_k)$. Then $(x_1,\dots, x_d)\mapsto(g_1,\dots, g_d)$ defines as well a $k'$-endomorphism of $k'(X_{k'})$. So a $k$-endomorphism $(x_1,\dots, x_n)\mapsto (g_1,\dots, g_n)$ of $k(X_k)$ extends uniquely to a $k'$-endomorphism of $k'(X_{k'})$. This yields the following observation:

\begin{lemma}
Let $X_k$ be a projective variety over a field $k$ and $\varphi\colon k\to k'$ a homomorphism of fields. Then $\varphi$ induces a natural injection of monoids $\Psi_{\varphi}\colon\Rat(X_k)\to\Rat(X_{k'})$. 
\end{lemma}

Recall that there are uncountably many finitely generated groups and thus in particular uncountably many finitely generated monoids. The following observation by de Cornulier shows that being a monoid of rational dominant transformations is in some sense a special property (cf. \cite{MR3160544} and \cite{cantatremark}).

\begin{proposition}\label{countablesubgroups}
There exist only countably many finitely generated isomorphism classes of monoids of rational dominant transformations.
\end{proposition}

\begin{proof}
Let $\Delta\subset\Rat(X_k)$ be a monoid of rational dominant transformations with a finite generating set $f_1,\dots, f_n\in\Delta$ , where $X_k$ is a projective variety defined over a field $k$. Denote by $S\subset k$ the finite set of coefficients necessary to define $X_k$ and the rational dominant transformations $f_i$. Let $k'$ be the field $\F_p(S)$, where $p=\car(k)$, or $\Q(S)$ if $\car(k)=0$. 

We consider the function field $k'(X_{k'})$ as a subfield of the function field  $k(X_{k})$. Note that the action of the elements of $\Delta$ on $k'(X_{k})$ preserves $k'(X_{k'})$  and that $f_{i_1}f_{i_2}\cdots f_{i_{k-1}}=f_{i_k}$ in $\Rat(X_k)$ if and only if $f_{i_1}f_{i_2}\cdots f_{i_{k-1}}=f_{i_k}$ in $\Rat(X_{k'})$. So without loss of generality we may assume $\Delta\subset\Rat(X_{k'})$, where $k'$ is a finitely generated field extension of some $\F_p$ or of $\Q$.

A rational dominant transformation of a given variety $X_k$ is defined by finitely many coefficients in $k$. So the cardinality of the set of all finitely generated monoids of rational dominant transformations of a variety $X_k$ is at most the cardinality of~$k$. 

Recall that the cardinality of the set of all finitely generated field extensions of $\F_p$ and $\Q$ is countable. Since a projective variety is defined by a finite set of coefficients, we obtain that there are only countably many isomorphism classes of projective varieties defined over a field $k'$ that is a finitely generated field extension of  $\F_p$ or $\Q$. The claim follows.
\end{proof}

\subsection{Intersection form}
Let $X_k$ be a smooth projective variety of dimension $d$ over an algebraically closed field $k$ and let $D$ be a Cartier divisor on $X_k$. The {\it Euler characteristic of $D$} is the integer
\[
\chi(X_k, D)=\sum_{i=0}^{\infty}(-1)^i\dim_k H^i(X_k, D).
\]
Let $D_1,\dots, D_d$ be Cartier divisors on $X_k$. The function 
\[
(m_1,\dots, m_d)\mapsto \chi(X_k, m_1D_1+\dots+m_dD_d)
\]
is a polynomial in $(m_1,\dots, m_d)$. The {\it intersection number} of $D_1,\dots, D_d$ is then defined to be the coefficient of the term $m_1m_2\cdots m_d$ in this polynomial and we denote it by $D_1\cdot D_2\cdots D_d$. One can show that intersection multiplicities are always integers and that the intersection form is symmetric and linear in all $d$ arguments. Moreover, if $D_1,\dots, D_d$ are effective and meet properly in a finite number of points, $D_1\cdots D_d$ is the number of points in $D_1\cap\cdots\cap D_d$ counted with multiplicities and the intersection form is the unique bilinear form with this property. For the details on this construction we refer to \cite{MR1841091} and the references in there.  The intersection number is preserved under linear equivalence, therefore it is well defined on classes of Cartier divisors. Note as well that an isomorphism between algebraically closed fields does not change the cohomology dimensions and hence that the intersection multiplicities are invariant under such an isomorphism.

Let $X_k$ be a smooth projective variety of dimension $d$ over an arbitrary field $k$ and let $D_1,\dots, D_d$ be Cartier divisors on $X_k$. Denote by $\overline{k}$ an algebraic closure of $k$. We define the intersection multiplicity $D_1\cdot D_2\cdots D_d$ as the intersection multiplicity of $D_1, D_2,\dots, D_d$ on $X_{\overline{k}}$ after base extension $k\to\overline{k}$. By the remark above, this does not depend on the choice of the algebraic closure $\overline{k}$. Every field isomorphism $k\to k'$ extends to an isomorphism between the algebraic closures of $k$ and $k'$, hence the intersection number is invariant under field isomorphisms. Since the intersection form is unique, it also does not change under a base extension $k\to k'$ between algebraically closed fields.

We summarize these properties in the following proposition:

\begin{proposition}\label{polarprop}
Let $X_k$ be a smooth projective variety of dimension $d$ over a field $k$. Then there exists a symmetric $d$-linear form on the group of divisors of $X_k$:
\[
\Div(X_k)\times\cdots\times \Div(X_k)\to\Z,\hspace{1.5mm}(D_1,\dots, D_d)\mapsto D_1\cdot D_2\cdots D_d,
\]
such that if $D_1,\dots, D_d$ are effective and meet properly in a finite number of points, $D_1\cdots D_d$ is the number of points in $D_1\cap\cdots\cap D_d$ counted with multiplicity. 
Moreover, this intersection form is invariant under base change $k\to k'$ of fields.
\end{proposition}

\subsection{Polarizations and degree functions}\label{polarizations}
A polarization on a smooth projective variety $X_k$ of dimension $d$ is an ample divisor class $H$. This implies in particular that $Y\cdot H^{d-1}>0$ for all effective divisors $Y$. 

Let $f\in\Rat(X_k)$ and denote by $\dom(f)$ the maximal open subset of $X_k$ on which $f$ is defined. The {\it graph} $\Gamma_f$ of $f$ is the closure of $\{(x,f(x)\mid x\in\dom(f))\}\subset X_k\times X_k$. Let $p_1$ and $p_2\colon\Gamma_f\to X$ be the natural projections on the first respectively second factor, then $f=p_2\circ p_1^{-1}$. Note that $p_1$ is birational. The {\it total transform} of a divisor $D$ under $f$ is the divisor ${p_1}_*(p_2^*D)$, where ${p_2}^*D$ is the pullback of $D$ as a Cartier divisor and ${p_1}_*(p_2^*D)$ is the pushforward of ${p_2}^*D$ as a Weil divisor. Note that if $k$ is algebraically closed and $D$ is effective, then $f^*D$ is the closure of all the points in $\dom(f)$ that are mapped to $D$ by $f$.

\subsection{Degrees in the case of complex varieties}\label{complex}
Let $X_{\C}$ be a smooth complex projective variety of dimension $d$. The following constructions and results can be found in \cite{MR2095471}. Recall that two divisors $D_1$ and $D_2$ are numerically equivalent if $D_1\cdot\gamma=D_2\cdot\gamma$ for all curves $\gamma$ on $X$. Denote by $N^1(X)$ the Neron-Severi group, which is the group of divisors modulo numerical equivalence.  The intersection number of divisors $D_1,\dots, D_d$ is invariant under numerical equivalence. By taking the first Chern class one can embed $N^1(X_{\C})$ into $H^2(X;\Z)_{t.f.}$, which is $H^2(X_{\C};\Z)$ modulo its torsion part. It turns out that $N^1(X_{\C})=H^2(X_{\C};\Z)_{t.f.}\cap H^{1,1}(X_{\C})$. Let $D_1,\dots, D_d$ be divisors and $\omega_1,\dots, \omega_d$ the corresponding $(1,1)$-forms, then
\[
D_1\cdot D_2\cdots D_d=\int_{X_{\C}}\omega_1\wedge \omega_2\wedge\cdots\wedge\omega_d.
\]
A divisor is ample if and only if the corresponding $(1,1)$-form is cohomologous to a K\"ahler form. Hence a polarization of a smooth complex projective variety can be seen as a K\"ahler class. In this section we will take this point of view.

Let $f$ be a birational transformation of $X$. With the help of currents, one can define the pullback $f^*\omega$ of a $(1,1)$-form $\omega$ (see \cite{MR2851870}). If $H$ is a divisor of $X$ and $\omega$ the corresponding $(1,1)$-form then $f^*\omega$ is  the form corresponding to $f^*H$. 

If we fix a polarization $\omega_{X_{\C}}$ of smooth complex projective variety $X_{\C}$, then the degree of a birational map $f\in\Bir(X_{\C})$ is defined by
\[
\deg_{\omega_{X_{\C}}}(f)=\int_{X_{\C}}f^*\omega_{X_{\C}}\wedge\omega_{X_{\C}}\wedge\cdots\wedge \omega_{X_{\C}}.
\]
So the degree does not depend on whether we look at an ample divisor class or its corresponding K\"ahler class.

In  \cite{MR2180409} the following is shown (see also \cite{MR2851870}): 

\begin{proposition}\label{dinhnguyen}
Let $X_{\C}$ be a smooth complex projective variety with polarization $\omega_{X_{\C}}$ and $f\in\Bir(X_{\C})$. There exists a constant $A>0$ such that
\[
A^{-1}\deg_{\omega_{X_{\C}}}(f^n)\leq \|(f^n)^*\|\leq A\deg_{\omega_{X_{\C}}}(f^n),
\] 
where $\|\,\cdot\,\|$ denotes any norm on $\End(H^{1,1}(X_{\C}))$.
\end{proposition}

The following corollary follows directly from Proposition \ref{dinhnguyen} and proves Proposition \ref{indeppolar}:

\begin{corollary}\label{cordinhnguyen}
Let $X_{\C}$ be a smooth complex projective variety, then 
\[
\dpol(f)=\limsup_{n\to\infty}\frac{\log(\|(f^n)^*\|)}{\log(n)}.
\]
In particular, $\dpol(f)$ does not depend on the choice of polarization. 
\end{corollary} 

\begin{proposition}\label{propfede}
Let $X_{\C}$ and $Y_{\C}$ be complex smooth projective varieties of dimension $d$, let $f\in\Bir(X_{\C})$, $g\in\Bir(Y_{\C})$ and $\pi\colon X\to Y$ a dominant morphism of finite degree $k$, such that $\pi\circ f=g\circ \pi$. Then $\dpol(f)=\dpol(g)$.
\end{proposition}

\begin{proof}
Let ${\omega_{X_{\C}}}$ be a K\"ahler form on $X_{\C}$ and $\omega_{Y_{\C}}$ be a K\"ahler form on $Y_{\C}$. We have
\begin{equation*}
\begin{split}
\deg_{\omega_{X_{\C}}}(f^n)&=\int_{X_{\C}}(f^n)^*\omega_{X_{\C}}\wedge\omega_{X_{\C}}^{d-1}=\deg(\pi)\int_{Y_{\C}}\pi_*(f^n)^*\omega_{X_{\C}}\wedge\pi_*\omega_{X_{\C}}^{d-1}\\
&=\deg(\pi)\int_{Y_{\C}}(g^n)^*\pi_*\omega_{X_{\C}}\wedge\pi_*\omega_{X_{\C}}^{d-1}\\
&\leq K_1\deg(\pi)\cdot\|(g^n)^*\|\cdot\|\pi_*\omega_{X_{\C}}\|\cdot\|\pi_*\omega_{X_{\C}}^{d-1}\|\leq K_2\cdot \|(g^n)^*\|,
\end{split}
\end{equation*}
for some positive constants $K_1, K_2$ not depending on $n$. By Corollary \ref{cordinhnguyen}, this yields $\dpol(f)\leq\dpol(g)$.

On the other hand, we have 
\begin{equation*}
\begin{split}
\deg_{\omega_{Y_{\C}}}(g^n)&=\int_{Y_{\C}}(g^n)^*\omega_{Y_{\C}}\wedge\omega_{Y_{\C}}^{d-1}=\frac{1}{\deg(\pi)}\int_{X_{\C}}\pi^*(g^n)^*\omega_{Y_{\C}}\wedge\pi^*\omega_{Y_{\C}}^{d-1}\\
&=\frac{1}{\deg(\pi)}\int_{X_{\C}}(f^n)^*\pi^*\omega_{Y_{\C}}\wedge\pi^*\omega_{Y_{\C}}^{d-1}\\
&\leq K_1'\frac{1}{\deg(\pi)}\|(f^n)^*\|\cdot\|\pi^*\omega_{Y_{\C}}\|\cdot\|\pi^*\omega_{Y_{\C}}^{d-1}\|\leq K_2'\cdot \|(f^n)^*\|,
\end{split}
\end{equation*}
for some positive constants $K_1', K_2'$ not depending on $n$ and therefore, by Corollary~\ref{cordinhnguyen},  $\dpol(g)\leq\dpol(f)$.

\end{proof}

\section{Proofs}
\subsection{Proof of Theorem \ref{rationalcount}}
Let $k$ be a field, $X_k$ a smooth projective variety defined over $k$, $H$ a polarization of $X_k$ and $\Delta\subset\Rat(X_k)$ a finitely generated monoid of rational dominant transformations with generating set $T$. Then $X_k$, $H$ and  $T$ are defined by a finite set $S$ of coefficients from $k$. Let $k'\subset k$ be the field $\F_p(S)$, where $p=\car(k)$, or $k'=\Q(S)$ if $\car(k)=0$. By Proposition \ref{polarprop}, the degree of elements in $\Gamma$ considered as rational dominant transformations of $X_k$ with respect to the polarization $H$ is the same as the degree of elements in $\Delta$ considered as rational dominant transformations of $X_{k'}$. So without loss of generality, we may assume that $\Delta$ is a submonoid of $\Rat(X_{k'})$, where $k'$ is a finite field extension of $\F_p$ or of~$\Q$. 

As in the proof of Proposition \ref{countablesubgroups}, we use the fact that there are only countably many isomorphism classes of such varieties. Polarizations and rational self maps are defined by a finite set of coefficients, so the cardinality of the set of all  $(k',X_{k'},H,T)$ up to isomorphism, where $k'$ is a finitely generated field extension of $\F_p$ or $\Q$, $X_{k'}$ a smooth projective variety over $k'$, $H$ a polarization of $X_{k'}$ and  a finite set of elements in $\Rat(X_{k'})$ is countable. It follows in particular that the set of all degree sequences is countable.

\subsection{Proof of Proposition \ref{aut1}}\label{proofaut1}
Let $k$ be a field and $f\in\Bir(\p_k^d)$ a birational map that is given by $[f_0:\dots:f_d]$, with respect to homogeneous coordinates $[x_0:\dots:x_d]$, where the $f_i$ are homogenous polynomials of the same degree without common factors. There are two important closed subsets of $\p_k^d$ associated to $f$, the indeterminacy locus $\Ind(f)$, consisting of all the points where $f$ is not defined and the exceptional divisor $\exc(f)$, the set of all the points where $f$ is not a local isomorphism. If $f$ is not an automorphism, the indeterminacy locus is a closed set of codimension $\geq 2$ and the exceptional divisor a closed set of codimension 1. Note that $\Ind(f)$ is exactly the set of points, where all the $f_i$ vanish. Let $X\subset \p_k^d$ be an irreducible closed set that is not contained in $\Ind(f)$. We denote by $f(X)$ the closure of $f(X\setminus \Ind(f))$ and we say that $f$ {\it contracts} $X$ if $\dim(f(X))<\dim(X)$. 

The following lemma is well known (see for example \cite{zbMATH00663864}):
\begin{lemma}\label{deglemma}
Let $k$ be a field and $g,f\in\Bir(\p_k^d)$. Then $\deg(f\circ g)\leq\deg(f)\deg(g)$ and $\deg(f\circ g)<\deg(f)\deg(g)$ if and only if $g$ contracts a hypersurface to a subset of $\Ind(f)$.
\end{lemma}

\begin{proof}
Let $f=[f_0:\dots:f_d]$ and $g=[g_0:\dots:g_d]$. Then $\deg(f\circ g)<\deg(f)\deg(g)$ if and only if the polynomials $f_0(g_0,\dots, g_d),\dots, f_d(g_0,\dots, g_d)$ have a non constant common factor $h\in k[x_0,\dots, x_d]$. Let $M\subset \p^d_k$ be the hypersurface defined by $h=0$. Then $f$ is not defined at $g(M)$. This implies that the codimension of $g(M)$ is $\geq 2$ and therefore that $g$ contracts $M$ to a subset of $\Ind(f)$. 

On the other hand, let $M$ be an irreducible component of a hypersurface that is contracted by $g$ to a subset $\Ind(f)$. Assume that $M$ is the zero set of an irreducible polynomial $h$. Since $g(M)\subset \Ind(f)$, we obtain that $f_0(g_0,\dots, g_d),\dots, f_d(g_0,\dots, g_d)$ vanish all on $M$ and therefore that $h$ divides $f_0(g_0,\dots, g_d),\dots, f_d(g_0,\dots, g_d)$. This implies $\deg(f\circ g)<\deg(f)\deg(g)$.
\end{proof}

In order to prove Proposition \ref{aut1}, we consider an element $f\in\aut(\A_k^d)$ as a birational transformation of $\p^d_k$ whose exceptional divisor is the hyperplane at infinity $H:=\p_k^d\setminus\A_k^d$ and whose indeterminacy points are contained in $H$.

Note that $\deg(f^d)=\deg(f)^d$ implies $\deg(f^l)=\deg(f)^l$ for $l=1,\dots, d$. We look at $f$ as an element of $\Bir(\p_k^d)$. If $f$ is an automorphism of $\p_k^d$, its degree is $1$ and the claim follows directly. Otherwise, $f$ contracts the hyperplane $H$. By Lemma \ref{deglemma}, $\deg(f^{l+1})=\deg(f^l)\deg(f)$ is equivalent to $f^l(H)$ not being contained in $\Ind(f)$. In particular, if $\deg(f^l)=\deg(f)^l$ and $f^l(H)=f^{l+1}(H)$ for some $l$ then $\deg(f^i)=\deg(f)^i$ for all $i\geq l$.

Let $1\leq l\leq d$. By Lemma \ref{deglemma}, $f^l(H)$ is not contained in $\Ind(f)$. Observe that $f^l(H)$ is irreducible and $f^{l+1}(H)\subset f^{l}(H)$. This implies that $\dim(f^{l+1}(H))\leq \dim(f^l(H))$ and $\dim(f^{l+1}(H))= \dim(f^l(H))$ if and only if $f^{l+1}(H)= f^{l}(H)$. It follows that the chain $H\supset f(H)\supset f^2(H)\supset\cdots$ becomes stationary within the first $d$ iterations. In particular, $f^d(H)=f^{d+i}(H)$ for all $i$ and therefore $\deg(f^i)=\deg(f)^i$ for all $i$.

\subsection{Proof of Theorem \ref{degaut} and Theorem \ref{degfinite}}
We start by proving Theorem~\ref{degaut}. Let $f\in\End(\A_k^d)$ be an endomorphism such that the sequence $\{\deg(f^n)\}$ is unbounded. Our first remark is that the elements of $\{f^n\}$ are linearly independent in the vector space of polynomial endomorphisms $\End(\A_k^d)$. Indeed, if for any $n$ we have
\[
f^n=\sum_{l<n}c_lf^l,
\]
for some $c_1,\dots, c_{n-1}\in k$. It follows by induction that  $\deg(f^{n+i})$ is smaller or equal to $\max_{l<n}\deg(f^l)$  for all~$i$ . 

Denote by $\End(\A^d)_{\leq K}$ the $k$-vector space of polynomial endomorphisms of degree $\leq K$, which has dimension $d\cdot\binom{d+K}{K}$. One calculates
\[
d\cdot\binom{d+K}{K}< \frac{1}{(d-1)!}(K+d)^{d}\leq \frac{(1+d/K)^{d}}{(d-1)!}K^{d}\leq C_dK^{d},
\]
where $C_d=\frac{(1+d)^{d}}{(d-1)!}$. Since the elements in $\{f^k\}\cap\End(\A^d)_{\leq K}$ are linearly independent, the cardinality of $\{f^k\}\cap\End(\A^d)_{\leq K}$ is at most $C_dK^{d-1}$. 

For the proof of Theorem \ref{degfinite} note that in the case of finite fields, there are only finitely many birational transformations of a given degree. If $f^l=f$ for some $l>1$, then $\{\deg(f^n)\}$ is bounded. There are $\binom{d+K}{K}$ monomials of degree $\leq K$. A birational transformation of degree $\leq K$ is given by $d+1$ polynomials of degree $\leq K$, so by $C(K,d)=(d+1)\binom{d+K}{K}$ coefficients from $\F_q$. Hence there are less than $q^{C(K,d)}$ birational transformations of degree $\leq K$. This proves the claim.

\subsection{Proof of Example \ref{exaut}}\label{proofexaut}
Let $d$ be an integer and $l=\lfloor d/2 \rfloor$. For $d=3$ the automorphism $f_3:=(x+z(y+xz), y+xz,z)$ from Example \ref{linearex} satisfies $\deg(f_3^n)\sim n$. Moreover, the first coordinate of $f_3^n$ is the coordinate with highest degree. Assume now that $d\geq 5$ and that we are given an automorphism $f_{d-2}\in\aut(\A^{d-2})$ such that $\deg(f_{d-2}^n)$ grows like $n^{l-1}$ and that the first coordinate of $f_{d-2}^n$ is the entry with highest degree. Let 
\[
f_d:=(x_1+x_3(x_2+x_1x_3), x_2+x_1x_3, f_{d-2}(x_3,\dots,x_{d})).
\]
One sees by induction that the degree of $f_d^n$ grows like $n^l$ and that the first coordinate of $f_{d}^n$ is the coordinate of $f_d^n$ with highest degree.

\subsection{Proof of Example \ref{oguiso}}\label{oguisoproof}
In this section we use the notation introduced in Example \ref{oguiso}.

Let $dz_1, dz_2, dz_3$ be a basis of $H^{1,0}(X)$. Then $d\bar{z_1},d\bar{z_2}, d\bar{z_3}$ is a basis of $H^{0,1}(X)$. The automorphism $f^n$ induces an action on both $H^{1,0}(X)$ and $H^{0,1}(X)$ whose norm grows like $n^2$. Since $H^{1,1}(X)=H^{1,0}(X)\otimes H^{0,1}(X)$, this implies that the norm of the induced action of $f^n$ on $H^{1,1}(X)$ grows like $n^4$. By Corollary \ref{cordinhnguyen}, we obtain $\dpol(f)=4$.

Denote by $\pi\colon X\to Y=X/s$ the quotient map. Let $\tilde{X}$ be a smooth projective variety, $\phi_2\colon\tilde{X}\to X$ a birational morphism and $\psi\colon\tilde{X}\to \tilde{Y}$ a dominant morphism such that the following diagram commutes:

\begin{center}
\begin{tikzcd}
\tilde{X} \arrow{r}{\psi} \arrow{d}{\phi_2}&\tilde{Y} \arrow{d}{\phi_1}
\\X \arrow{r}{\pi}&Y. 
\end{tikzcd}
\end{center}
Note that $\psi$ is generically finite. By Proposition \ref{propfede}, we have $\dpol(\phi_2^{-1}\circ f\circ\phi_2)=\dpol(f)=4$ and hence, again by Proposition \ref{propfede}, $\dpol(\tilde{f})=\dpol(\psi\circ \phi_2^{-1}\circ f\circ\phi_2\circ\psi^{-1})=4$, which proves the claim of Example \ref{oguiso}.

\section{Remarks}

\subsection{Other degree functions}
One can define more general degree functions. Let $X_k$ be a smooth projective variety over a field $k$ with polarization $H$ and $1\leq l\leq d-1$, then
\[
\deg_H^l(f):=(f^*H)^l\cdot H^{d-l}.
\]
These degree functions play an important role in dynamics. In characteristic 0, we still have $\deg_H^l(fg)\leq C\deg_H^l(f)\deg_H^l(g)$ for a constant $C$ not depending on $f$ and $g$ (see \cite{MR2180409} and \cite{Truong:2015aa} for generalizations to fields of positive characteristic). 

Our proof of Theorem \ref{rationalcount} works as well if we replace the function $\deg_H$ by $\deg_H^l$ for any $l$. Let $\Gamma\subset\Bir(X_k)$ be a finitely generated group of birational transformations with a finite symmetric set of generators $S$. We define 
$D^l_{S, H}\colon\Z^+\to\Z^+$
by
$D^l_{S,H}(n):=\max_{\gamma\in B_S(n)}\{\deg_H(\gamma)\}$ and we call a map $\Z^+\to\Z^+$ that can be realized for some $(X_k, H, \Gamma, S,l)$ as such a function a {\it general degree sequence}.

\begin{theorem}
The set of all general degree sequences is countable.
\end{theorem}

\begin{proof}
Analogous to the proof of Theorem \ref{rationalcount}.
\end{proof}

\bibliographystyle{amsalpha}
\bibliography{/Users/christian/Dropbox/Literatur/bibliography_cu}

\end{document}